\documentclass[11pt]{article} 

\usepackage{amsmath}
\usepackage{amssymb}
\usepackage{color}
\usepackage[T1]{fontenc}
\usepackage[latin1]{inputenc}
\def\r{\rightarrow}

\usepackage{enumerate}

\newcommand{\fdem}{\hspace*{\fill}~$\Box$\par\endtrivlist\unskip}

\renewcommand{\L}{\mathbb{L}}

\newcommand{\N}{\mathbb{N}}     

\newcommand{\R}{\mathbb{R}}     
     
\newcommand{\C}{\mathbb{C}}

\renewcommand{\r}{\mathop{\rightarrow}}

\newcommand{\cB}{\mbox{$\cal B$}}

\newcommand{\cU}{\mbox{$\cal U$}}

\newtheorem{theo}{Theorem}
\newtheorem{pro}{Proposition}
\newenvironment{proof}[1]{\textit{Proof#1.\,}}{\fdem}
\newtheorem{lem}{Lemma}
\newtheorem{rem}{Remark}

\newtheorem{cor}{Corollary}

\newtheorem{aex}{Example}[section]

\title{A computable bound of the essential spectral radius of finite range Metropolis-Hastings kernels}

\author{Loïc HERV\'E, and James LEDOUX \footnote{INSA de Rennes, IRMAR, F-35042, France; CNRS, UMR 6625, Rennes, F-35708, France; Université Européenne de Bretagne, France. \{Loic.Herve,James.Ledoux\}@insa-rennes.fr}
}

\begin{document}

\maketitle
\begin{abstract}
Let $\pi$ be a positive continuous target density on $\R$. Let $P$ be  the Metropolis-Hastings operator on the Lebesgue space $\L^2(\pi)$ corresponding to a proposal Markov kernel $Q$ on $\R$. When using the quasi-compactness method to estimate the spectral gap of $P$, a mandatory first step is to obtain an accurate bound of the essential spectral radius  $r_{ess}(P)$ of $P$. In this paper a computable bound of $r_{ess}(P)$ is obtained  under the following assumption on the proposal kernel: $Q$ has a bounded continuous density $q(x,y)$ on $\R^2$ satisfying the following finite range assumption : $|u| > s \, \Rightarrow\,  q(x,x+u) = 0$ (for some $s>0$). This result is illustrated with Random Walk Metropolis-Hastings kernels. 
\end{abstract}
\begin{center}
AMS subject classification : 60J10, 47B07

Keywords : Markov chain operator, Metropolis-Hastings algorithms, Spectral gap
\end{center}

\section{Introduction}
Let $\pi$ be a positive distribution density on $\R$. Let $Q(x,dy) = q(x,y)dy$ be a Markov kernel on $\R$. Throughout the paper we assume that $q(x,y)$ satisfies the following finite range assumption: there exists $s>0$ such that   
\begin{equation} \label{Ass-sup-fini}
|u| > s \ \Longrightarrow\  q(x,x+u) = 0.
\end{equation}
Let $T(x,dy) = t(x,y)dy$ be the nonnegative kernel on $\R$ given by 
\begin{equation} \label{def-t}
t(x,y) := 
     \min\left(q(x,y)\, ,\, \frac{\pi(y)\, q(y,x)}{\pi(x)}\right)
\end{equation}
and define the associated Metropolis-Hastings kernel:  
\begin{equation} \label{def-P-noyau}
P(x,dy) := r(x)\, \delta_x(dy) + T(x,dy) \qquad \text{with }\ r(x) := 1 - \int_\R t(x,y)\, dy,
\end{equation}
where $\delta_x(dy)$ denotes the Dirac distribution at $x$. The associated Markov operator is still denoted by $P$, that is we set for every bounded measurable function $f : \R\r\C$ : 
\begin{equation} \label{def-P}
\forall x\in\R,\quad (Pf)(x) = r(x) f(x) + \int_\R f(y)\, t(x,y)\, dy.
\end{equation}
In the context of Monte Carlo Markov Chain methods, the kernel $Q$ is called the proposal Markov  kernel. 
We denote by $(\L^2(\pi),\|\cdot\|_{2})$ the usual Lebesgue space associated with the probability measure $\pi(y)dy$. For convenience, $\|\cdot\|_{2}$ also denotes the operator norm on $\L^2(\pi)$, namely: if $U$ is a bounded linear operator on $\L^2(\pi)$, then $\|U\|_2 := \sup_{\|f\|_2=1}\|Uf\|_2$. Since
\begin{equation} \label{bal-eq}
t(x,y)\pi(x) = t(y,x)\pi(y), 
\end{equation}
we know that $P$ is reversible with respect to $\pi$ and that $\pi$ is $P$-invariant (e.g.~see \cite{RobRos04}). Consequently $P$ is a self-adjoint operator on $\L^2(\pi)$ and $\|P\|_2=1$. 
Now define the rank-one projector $\Pi$ on $\L^2(\pi)$ by 
$$\Pi f := \pi(f) 1_{\R}\quad \text{with }\ \pi(f):=\int_\R f(x)\, \pi(x)\, dx.$$ 
Then the spectral radius of $P-\Pi$ equals to $\|P-\Pi\|_2$ since $P-\Pi$ is self-adjoint, and $P$ is said to have the spectral gap property on  $\L^2(\pi)$ if 
$$\varrho_2 \equiv \varrho_2(P) := \|P-\Pi\|_2 <1.$$ 
In this case the following property holds: 
\begin{equation} \label{ineg-gap-gene} 
\forall n\geq1, \forall f\in\L^2(\pi),\quad \|P^nf - \Pi f\|_2 \leq \varrho_2^{\, n}\, \|f\|_2. \tag{SG$_2$}
\end{equation}
The spectral gap property on $\L^2(\pi)$ of a Metropolis-Hastings kernel is of great interest, not only due to the explicit geometrical rate given by (\ref{ineg-gap-gene}), but also since it ensures that a central limit theorem (CLT) holds true for additive functional of the associated Metropolis-Hastings Markov chain  under the expected second-order moment conditions, see \cite{RobRos97}. Furthermore, the rate of convergence in the CLT is O$(1/\sqrt{n})$ under third-order moment conditions (as for the independent and identically distributed models), see details in \cite{HerPen10,FerHerLed10}. 

The quasi-compactness approach can be used to compute the rate $\varrho_2(P)$. This method is based on the notion of essential spectral radius. Indeed, first recall that the essential spectral radius of $P$ on $\L^2(\pi)$, denoted by $r_{ess}(P)$, is defined by (e.g.~see \cite{Wu04} for details): 
\begin{equation} \label{def-r-ess}
r_{ess}(P) := \lim_n(\inf \|P^n-K\|_2)^{1/n}
\end{equation}
where the above infimum is taken over the ideal of compact operators $K$ on $\L^2(\pi)$. Note that the spectral radius of $P$ is one. Then $P$ is said to be quasi-compact on $\L^2(\pi)$ if $r_{ess}(P)<1$. 
Second, if $r_{ess}(P) \leq \alpha$ for some $\alpha\in(0,1)$, then $P$ is quasi-compact on $\L^2(\pi)$, and the following properties hold: for every real number $\kappa$ such that $\alpha <\kappa<1$, the set $\cU_\kappa$ 
of the spectral values $\lambda$ of $P$ satisfying $\kappa\leq |\lambda| \leq1$ is composed of finitely many eigenvalues of $P$, each of them having a finite multiplicity (e.g.~see \cite{Hen93} for details). Third, if $P$  is quasi-compact on $\L^2(\pi)$ and satisfies usual  aperiodicity and irreducibility conditions (e.g.~see \cite{MeyTwe93}), then  $\lambda=1$ is the only spectral value of $P$ with modulus one and $\lambda=1$ is a simple eigenvalue of $P$, so that $P$ has the spectral gap property on  $\L^2(\pi)$. Finally the following property holds: either $\varrho_2(P)=\max\{|\lambda|, \lambda\in\cU_\kappa,\, \lambda\neq 1\}$ if $\cU_\kappa\neq\emptyset$, or $\varrho_2(P)\leq \kappa$ if $\cU_\kappa = \emptyset$.

This paper only focusses on the preliminary central step of the previous spectral method, that is to find an accurate bound of $r_{ess}(P)$. 
More specifically, 
we prove that, if the target density $\pi$ is positive and continuous on $\R$, and if the proposal kernel $q(\cdot,\cdot)$ is bounded continuous on $\R^2$ and satisfies (\ref{Ass-sup-fini}) for some $s>0$, then 
\begin{equation} \label{bound-r-ess-L2}
r_{ess}(P) \leq \alpha_a \qquad \text{with }\ \alpha_a:=\max(r_a\, , \, r_a'+\beta_a)
\end{equation}
where, for every $a>0$, the constants $r_a, r_a'$ and $\beta_{a}$ are defined by: 
\begin{equation} \label{ra-r'a-betaa}
r_a := \sup_{|x|\leq a} r(x),\ \ r_a' := \sup_{|x| > a} r(x),\ \ \beta_{a} := \int_{-s}^s \ \sup_{|x| > a}\sqrt{t(x,x+u)\, t(x+u,x)}\ du. 
\end{equation}
This result is illustrated in Section~\ref{sec-ress-L2} with Random Walk Metropolis-Hastings (RWMH) kernels for which the proposal Markov kernel is of the form 
$Q(x,dy):=\Delta(|x-y|)\, dy$, where $\Delta : \R\r [0,+\infty)$ is an even continuous and compactly supported function. 

In \cite{AtcPer07} the quasi-compactness of $P$ on $\L^2(\pi)$ is proved to hold provided that 1) the essential supremum of the rejection probability $r(\cdot)$ with respect to $\pi$ is bounded away from unity; 2) the operator $T$ associated with the kernel $t(x,y)dy$ is compact on $\L^2(\pi)$. Assumption~1) on the rejection probability $r(\cdot)$ is a necessary condition for $P$ to have the spectral gap property (\ref{ineg-gap-gene}) (see \cite{RobTwe96}). But this condition, which is quite generic from the definition of $r(\cdot)$ (see Remark~\ref{rk-r-2}), is far to be sufficient for $P$ to satisfy (\ref{ineg-gap-gene}). The compactness Assumption 2) of \cite{AtcPer07} is quite restrictive, for instance it is not adapted for random walk Metropolis-Hastings kernels. Here this compactness assumption is replaced by the condition $r_a'+\beta_a < 1$. As shown in the examples of Section~\ref{sec-ress-L2}, this condition is adapted to RWMH. 

 In the discrete state space case, 
a  bound for $r_{ess}(P)$ similar to (\ref{bound-r-ess-L2}) has been obtained in \cite{HerLedJAP16}. Next a bound of the spectral gap $\varrho_2(P)$ has been derived in \cite{HerLedJAP16} from a truncation method for  which the control of the essential spectral radius of $P$ is a central step. It is expected that, in the continuous state space case, 
the bound (\ref{bound-r-ess-L2}) will provide a similar way to compute the spectral gap $\varrho_2(P)$ of $P$. This issue, which is much more difficult than in the discrete case, is not addressed in this work. 

\section{An upper bound for the essential spectral radius of $P$} 
\label{sec-ress-L2}
%
Let us state the main result of the paper.
\begin{theo} \label{theo-r-ess}
Assume that 
\vspace*{-3mm}
\begin{enumerate}[(i)]
	\item $\pi$ is positive and continuous on $\R$; 
	\item $q(\cdot,\cdot)$ is bounded and continuous on $\R^2$, and satisfies the finite range assumption (\ref{Ass-sup-fini}).
\end{enumerate}
\vspace*{-3mm}
For $a>0$, set $\alpha_a := \max(r_a\, , \, r_a'+\beta_a)$, where the  constants $r_a, r_a'$ and $\beta_{a}$ are defined in (\ref{ra-r'a-betaa}). 
Then 
\begin{equation} \label{70}
	\forall a>0, \quad r_{ess}(P) \leq \alpha_a.
\end{equation}
\end{theo}
Theorem~\ref{theo-r-ess} is proved in Section~\ref{sub-pf-ress}  from Formula (\ref{def-r-ess}) by using a suitable decomposition of the iterates $P^n$  involving some Hilbert-Schmidt operators.
\begin{rem} \label{rem-V-geo}
Assume that the assumptions $(i)$-$(ii)$ of Theorem~\ref{theo-r-ess} hold. Then, if there exists some $a>0$ such that $\alpha_a <1$, $P$ is quasi-compact on $\L^2(\pi)$. Suppose moreover that the proposal Markov kernel $Q(x,dy)$ satisfies usual irreducibility and aperiodicity conditions. Then $P$ has the spectral gap property on  $\L^2(\pi)$. Actually, if $q$ is symmetric (i.e.~$q(x,y)=q(y,x)$), it can be easily proved that, under the condition $r_a'+\beta_a<1$, $P$ satisfies the so-called drift condition with respect to $V(x):=1/\sqrt{\pi(x)}$, so that $P$ is $V$-geometrically ergodic, that is $P$ has the spectral gap property on the space $(\cB_V,\|\cdot\|_V)$ composed of the functions $f : \R\r\R$ such that $\|f\|_V := \sup_{x\in\R} |f(x)|/V(x) < \infty$. If furthermore $\int_\R\sqrt{\pi(x)}\, dx < \infty$, then the spectral gap property of $P$ on $\L^2(\pi)$ can be deduced from the $V$-geometrical ergodicity since $P$ is reversible (see \cite{RobRos97,Bax05}). However this fact does not provide a priori any 
precise bound 
on the essential spectral radius of $P$ on $\L^2(\pi)$. Indeed, mention that the results \cite[Th.~5.5]{Wu04} provide a comparison between $r_{ess}(P)$ and $r_{ess}(P_{|{\cal B}_V})$, but unfortunately, to the best of our knowledge, no accurate bound of $r_{ess}(P_{|{\cal B}_V})$ is known for Metropolis-Hasting kernels. In particular note that the general bound of $r_{ess}(P_{|{\cal B}_V})$ given in \cite[Th.~5.2]{HerLed14a} is of theoretical interest but is not precise, and that the more accurate bound of $r_{ess}(P_{|{\cal B}_V})$ given in \cite[Th.~5.4]{HerLed14a} cannot be used here since in general no iterate of $P$ is compact from $\cB_0$ to $\cB_V$, where $\cB_0$ denotes the space of bounded measurable functions $f : \R\r\R$ equipped with the supremum norm. Therefore, the $V$-geometrical ergodicity of $P$ is not discussed here since the purpose is to bound the essential spectral radius of $P$ on $\L^2(\pi)$.
\end{rem}
\begin{rem} \label{rk-r-1} 
If $\pi$ and $q$ satisfy the assumptions $(i)$-$(ii)$ of Theorem~\ref{theo-r-ess}, and if moreover $q$ satisfies the following mild additional condition 
\begin{equation} \label{add-cond-q}
\forall x\in\R,\ \exists y\in[x-s,x+s],\quad q(x,y)\, q(y,x) \neq 0,
\end{equation}
then, for every $a>0$, we have $r_a<1$, so that 
the quasi-compactness of $P$ on $\L^2(\pi)$ holds
provided that there exists some $a>0$ such that  $r_a'+\beta_a <1$. Note that Condition~(\ref{add-cond-q}) is clearly fulfilled if $q$ is symmetric. To prove the previous assertion on $r_a$, observe that $r(\cdot)$ is continuous on $\R$ (use Lebesgue's theorem). Consequently, if $r_a=1$ for some $a>0$, then $r(x_0)=1$ for some $x_0\in[-a,a]$, but this is impossible from the definition of $r(x_0)$ and Condition~(\ref{add-cond-q}). 
\end{rem}
\begin{rem} \label{rk-r-2} 
Actually, under the assumptions $(i)$-$(ii)$ of Theorem~\ref{theo-r-ess}, the fact that $r_a<1$ for every $a>0$, and even the stronger property 
$\sup_{x\in\R} r(x)<1$, seem to be quite generic. For instance, if $q$ is of the form $q(x,y)= \Delta(|x-y|)$ for some function $\Delta$ and if there exists $\theta>0$ such that $\pi$ is increasing on $(-\infty,-\theta]$ and decreasing on $[\theta,+\infty)$, then $\sup_{x\in\R} r(x)<1$. Thus, for every $a>0$, we have $r_a<1$ and $r_a'<1$. Indeed, first observe that $r_a<1$ for every $a>0$ from Remark~\ref{rk-r-1}. Consequently, if $\sup_{x\in\R} r(x)=1$, then there exists 
$(x_n)_n\in\R^{\N}$ such that $\lim_n|x_n| = +\infty$ and $\lim_n r(x_n)=1$. Let us prove that this property is impossible under our assumptions. To simplify, suppose that $\lim_n x_n = +\infty$. Then, from the definition of $r(\cdot)$, from our assumptions on $q$, and finally from Fatou's Lemma, it follows that, for almost every $u\in[-s,s]$ such that $\Delta(u)\neq 0$, we have  $\liminf_n\min(1,\pi(x_n+u)/\pi(x_n))=0$. But this is impossible since, if $u\in[-s,0]$ and $x_n \geq \theta+s$, then  $\pi(x_n+u) \geq \pi(x_n)$.  
\end{rem}

Theorem~\ref{theo-r-ess} is illustrated with symmetric proposal Markov kernels of the form 
$$Q(x,dy):=\Delta(x-y)\, dy$$ 
where $\Delta : \R\r [0,+\infty)$ is : 1) an even continuous function ; 2) assumed to be compactly supported on $[-s,s]$ and positive on $(-s,s)$ for some $s>0$. 
Then $q(x,y):=\Delta(x-y)$ satisfies (\ref{Ass-sup-fini}) and $t(\cdot,\cdot)$ is given by 
$$\forall u\in [-s,s],\quad t(x,x+u) := \Delta(u)\min\left(1\, ,\, \frac{\pi(x+u)}{\pi(x)}\right).$$
\begin{cor} \label{cor-lim}
Assume that $q(x,y):=\Delta(x-y)$
with $\Delta(\cdot)$ satisfying the above assumptions and that $\pi$ is an even positive continuous distribution density such that the following limit exists: 
\begin{equation} \label{tauu}
\forall u\in[0,s],\quad \tau(u) := \lim_{x\r+\infty}\frac{\pi(x+u)}{\pi(x)} \in [0,1].
\end{equation}
Assume that the set $\{u\in[0,s] : \tau(u) \neq 1\}$ has a positive Lebesgue-measure. 
Then 
$P$ is quasi-compact on $\L^2(\pi)$ with 
$$r_{ess}(P) \leq \alpha_\infty:=\max\big(r_\infty\, , \, \gamma_\infty\big)\, < 1 \quad \text{where }\ \gamma_\infty := 1 - \int_{0}^s \Delta(u)\big(1 - \tau(u)^{1/2}\big)^2 \, du.$$
\end{cor}
\begin{proof}{}
We know from Theorem~\ref{theo-r-ess} that, for any $a>0$, $r_{ess}(P) \leq \max\big(r_a\, , \, r_a'+\beta_a\big)$ with $r_a := \sup_{|x|\leq a} r(x)$, $r_a' := \sup_{|x| > a} r(x)$. It is easily checked that 
$$\beta_a = \int_{-s}^s \Delta(u)\sup_{|x| > a}\min\left(\sqrt{\frac{\pi(x+u)}{\pi(x)}},\sqrt{\frac{\pi(x)}{\pi(x+u)}}\right)\, du.$$
Note that 
$$\forall x\in\R,\quad r(x) = 1 - \int_{-s}^s \Delta(u) \min\bigg(1,\frac{\pi(x+u)}{\pi(x)}\bigg)\, du.$$
For $u\in[-s,0]$, $\tau(u)$ is defined as in (\ref{tauu}). Then 
$$\forall u\in[-s,s],\quad \tau(u) = \lim_{y\r+\infty}\frac{\pi(y)}{\pi(y-u)} = \frac{1}{\tau(-u)}$$
with the convention $1/0=+\infty$. Thus, for every  $u\in[-s,0]$, we have 
$\tau(u)\in[1,+\infty]$. Moreover we obtain for every $u\in [-s,s]$: 
$$\lim_{x\r-\infty}\frac{\pi(x+u)}{\pi(x)} = \tau(-u).$$
since $\pi$ is an even function. We have for every $a>0$ 
$$r_a' = 1 - \min\left(\inf_{x<-a}\int_{-s}^s \Delta(u) \min\bigg(1,\frac{\pi(x+u)}{\pi(x)}\bigg)\, du\ ,\ \inf_{x>a} \int_{-s}^s \Delta(u) \min\bigg(1,\frac{\pi(x+u)}{\pi(x)}\bigg)\, du\right).$$
Moreover it follows from dominated convergence theorem and from the above remarks that 
$$\lim_{x\r\pm\infty} \int_{-s}^s \Delta(u) \min\bigg(1,\frac{\pi(x+u)}{\pi(x)}\bigg)\, du = \int_{-s}^s \Delta(u) 
\min\big(1,\tau(\pm u)\big)\, du$$
from which we deduce that 
\begin{eqnarray*}
\lefteqn{r_\infty' := \lim_{a\r+\infty}r_a'}\\
 &=&  1 - \min\left(\int_{-s}^s \Delta(u) 
\min\big(1,\tau(-u)\big)\, du\ ,\ \int_{-s}^s \Delta(u) 
\min\big(1,\tau(u)\big)\, du \right) \\ 
&=&  1 - \int_{-s}^s \Delta(u) \min\big(1,\tau(u)\big)\, du \qquad \text{(since $\Delta$ is an even function)}\\
&=& 1 - \int_{-s}^0 \Delta(u) \, du - \int_{0}^s \Delta(u) \tau(u)\, du \quad \text{(since } \tau(u)\leq 1\text{ for } u \in[0,s], \tau(u)\geq 1 \text{ for } u \in[-s,0]) \\
&=& 1 - \int_{0}^s \Delta(u)\, \big[1+\tau(u)\big] \, du. 
\end{eqnarray*}
Note that, for every $a>0$, we have $r_a<1$ from Remark~\ref{rk-r-1}. Moreover $r_\infty'\leq 1/2$ from the last equality. Thus $r_\infty := \sup_{x\in\R} r(x)<1$.  Next we obtain for every $a>0$ 
$$
\beta_a = \int_{-s}^s \Delta(u)\max\left[\sup_{x<-a}\min\left(\sqrt{\frac{\pi(x+u)}{\pi(x)}},\sqrt{\frac{\pi(x)}{\pi(x+u)}}\right),\ \sup_{x>a}\min\left(\sqrt{\frac{\pi(x+u)}{\pi(x)}},\sqrt{\frac{\pi(x)}{\pi(x+u)}}\right)\right]\, du$$
and again we deduce from dominated convergence theorem and from the above remarks that 
\begin{eqnarray*}
\beta_\infty := \lim_{a\r+\infty}\beta_a 
&=& \int_{-s}^s \Delta(u) \max\bigg[\min\left(\tau(-u)^{1/2},\frac{1}{\tau(-u)^{1/2}}\right)\, ,\, \min\left(\tau(u)^{1/2},\frac{1}{\tau(u)^{1/2}}\right)\bigg]\, du \\
&=& \int_{-s}^s \Delta(u) \min\left(\tau(u)^{1/2},\frac{1}{\tau(u)^{1/2}}\right)\, du \qquad \text{(since $\tau(-u) = \frac{1}{\tau(u)}$)}\\
&=& \int_{-s}^0 \Delta(u)\, \tau(u)^{-1/2}\, du + \int_{0}^s \Delta(u)\, \tau(u)
^{1/2}\, du \\
&=& \int_{-s}^0 \Delta(u)\, \tau(-u)^{1/2}\, du + \int_{0}^s \Delta(u)\, \tau(u)
^{1/2}\, du \\
&=& 2\int_{0}^s \Delta(u)\, \tau(u)^{1/2}\, du. 
\end{eqnarray*}
Thus 
$$r_\infty'+ \beta_\infty = 1 - \int_{0}^s \Delta(u)\big[1+ \tau(u) - 2\tau(u)^{1/2}\big]\, du = 1 - \int_{0}^s \Delta(u)\, \big(1 - \sqrt{\tau(u)}\big)^2\, du\ <1$$
since by hypothesis the set $\{u\in[0,s] : \tau(u) \neq 1\}$ has a positive Lebesgue-measure. 

Since $r_{ess}(P) \leq \max\big(r_a\, , \, r_a'+\beta_a\big)$ holds for every $a>0$, we obtain that $r_{ess}(P) \leq \max\big(r_\infty\, , \, r_\infty'+\beta_\infty\big)<1$. Thus $P$ is quasi-compact on $\L^2(\pi)$. 
\end{proof}

\begin{aex} [Laplace distribution]
Let $\pi(x) = e^{-|x|}/2$ be the Laplace distribution density, and set $q(x,y):=\Delta(x-y)$ with $\Delta(u):= (1-|u|)\, 1_{[-1,1]}(u)$. Then 
$$\forall u\in[0,1],\quad \tau(u) := \lim_{x\r+\infty}\frac{\pi(x+u)}{\pi(x)} = e^{-u}.$$
Then 
$$\gamma_\infty = 1 - \int_{0}^1 (1-u)\big(1 - e^{-u/2}\big)^2 \, du = 8\, e^{-1/2} - e^{-1} - 7/2. $$
From Corollary~\ref{cor-lim}, $P$ is quasi-compact on $\L^2(\pi)$ with $r_{ess}(P) \leq \max\big(1-1/e\, , \, 8\, e^{-1/2} - e^{-1} - 7/2\big) = 8\, e^{-1/2} - e^{-1} - 7/2 \approx 0.9843$ since $r_\infty := \sup_{x\in\R} r(x) \le 1-1/e$. 
\end{aex}
\begin{aex} [Gauss distribution]
Let $\pi(x) = e^{-x^2/2}/\sqrt{2\pi}$ be the Gauss distribution density, and set $q(x,y):=\Delta(|x-y|)$ with $\Delta(u):= (1-|u|)\, 1_{[-1,1]}(u)$. Then 
$$\forall u\in(0,1],\quad \tau(u) := \lim_{x\r+\infty}\frac{\pi(x+u)}{\pi(x)} = 0,$$
so that 
$$\gamma_\infty = 1 - \int_{0}^1 (1-u) \, du = \frac{1}{2}.$$ 
From Corollary~\ref{cor-lim}, $P$ is quasi-compact on $\L^2(\pi)$ with $r_{ess}(P) \leq \max\big(0.156\, , \, 0.5\big) =0.5$ since $r_\infty \le 1-e^{-1/2}-e^{1/8}\int_0^1 (1-u)e^{-(u+1)^2/2}du\le 0.156$. 
\end{aex}

 In view of the quasi-compactness approach presented in Introduction for computing the rate $\varrho_2(P)$ in (\ref{ineg-gap-gene}), the bound $r_{ess}(P) \leq 0.5$ obtained for Gauss distribution (for instance) implies that, for every $\kappa\in(0.5,1)$,  the set of the spectral values $\lambda$ of $P$ on $\L^2(\pi)$ satisfying $\kappa\leq |\lambda| \leq 1$ is composed of finitely many eigenvalues of finite multiplicity. Moreover, from aperiodicity and irreducibility, $\lambda=1$ is the only eigenvalue of $P$ with modulus one and it is a simple eigenvalue of $P$. Consequently the spectral gap property (\ref{ineg-gap-gene}) holds with $\varrho_2(P)$ given by 
\vspace*{-4mm}
\begin{itemize}
	\item $\varrho_2(P)=\max\big\{|\lambda|, \lambda\in\cU_\kappa,\, \lambda\neq 1\big\}$ if $\cU_\kappa\neq\emptyset$, 
	\item $\varrho_2(P)\leq \kappa$ if $\cU_\kappa = \emptyset$ (in particular, if for every $\kappa\in(0.5,1)$ we have $\cU_\kappa = \emptyset$, then we could conclude that $\varrho_2(P)\leq 0.5$). 
\end{itemize}
\vspace*{-4mm}

The numerical computation of the eigenvalues $\lambda\in\cU_\kappa$, $\lambda\neq 1$,  is a difficult issue. Even to know whether the set $\cU_\kappa\setminus\{1\}$ is empty or not seems to be difficult. In the discrete state space case (i.e~$P=(P(i,j))_{i,j\in\N}$), this problem has been solved by using a weak perturbation method involving some finite truncated matrices derived from $P$ (see \cite{HerLedJAP16}). In the continuous state space case, a perturbation method could be also considered, but it raises a priori difficult theoretical and numerical issues.

\section{Proof of Theorem~\ref{theo-r-ess}} \label{sub-pf-ress}

For any bounded linear operator $U$ on $\L^2(\pi)$ we define 
$$\forall f\in\L^2(\pi),\quad U_af := 1_{[-a,a]}\cdot Uf \quad \text{and} \quad U_{a^c}f := 1_{\R\setminus[-a,a]}\cdot Uf.$$
Obviously $U_a$ and $U_{a^c}$ are bounded linear operators on $\L^2(\pi)$, and 
$U = U_a + U_{a^c}$. Define $Rf = r f$ with function $r(\cdot)$ given in (\ref{def-P-noyau}). Recall that $T$ is the operator associated with kernel $T(x,dy)= t(x,y)dy$. Then the M-H kernel $P$ defined in (\ref{def-P}) writes as follows: 
$$P = R+T = R_a + R_{a^c} + T_a + T_{a^c}$$
with
$R_{a^c}R_a =R_aR_{a^c} = 0$ and $R_a T_{a^c}=0$.
\begin{lem} \label{Ta-comp} The operators $T_a, T_{a^c} R_a$ and $(R_{a^c} + T_{a^c})^n R_a$ for any $n\geq 1$ are compact on $\L^2(\pi)$. 
\end{lem}
\begin{proof}{}
Using the detailed balance equation (\ref{bal-eq}),  we obtain for any $f\in\L^2(\pi)$ 
\begin{eqnarray*}
(T_af)(x) &=& 1_{[-a,a]}(x)\int_\R f(y)\, t(x,y)\, dy = \int_\R f(y)\, 1_{[-a,a]}(x)\, \frac{t(x,y)}{\pi(y)}\, \pi(y)\, dy \\
&=& \int_\R f(y)\, t_a(x,y) \, \pi(y)\, dy \qquad \text{with }\ t_a(x,y) := 1_{[-a,a]}(x)\, \frac{t(y,x)}{\pi(x)}. 
\end{eqnarray*} 
Function $q(\cdot,\cdot)$ is supposed to be bounded on $\R^2$, so is $t(\cdot,\cdot)$. From $\inf_{|x|\leq a} \pi(x)>0$ it follows that $t_a(\cdot,\cdot)$ is bounded on $\R^2$. Consequently $t_a\in\L^2(\pi\otimes\pi)$, so that $T_a$ is a Hilbert-Schmidt operator on $\L^2(\pi)$. In particular $T_a$ is compact on $\L^2(\pi)$. 

Now observe that 
$$
(T_{a^c}R_af)(x) = 1_{\R\setminus[-a,a]}(x)\int_\R 1_{[-a,a]}(y)\, r(y)\, f(y)\, t(x,y)\, dy = \int_\R f(y)\, k_a(x,y)\, \pi(y)\, dy
$$
where $k_a(x,y) := 1_{[-a,a]}(y)\, 1_{\R\setminus[-a,a]}(x)\,  r(y)\, t(x,y)\, \pi(y)^{-1}$. Then $T_{a^c}R_a$ is a Hilbert-Schmidt operator on $\L^2(\pi)$ since $k_a(\cdot,\cdot)$ is bounded on $\R^2$ from our assumptions. Thus $T_{a^c}R_a$ is compact on $\L^2(\pi)$. 

Let us prove by induction that $(R_{a^c} + T_{a^c})^n R_a$ is compact on $\L^2(\pi)$ for any $n \geq 1$. For $n=1$, $(R_{a^c} + T_{a^c}) R_a= T_{a^c} R_a $ is compact. 
Next, 
$$(R_{a^c} + T_{a^c})^n R_a = (R_{a^c} + T_{a^c})^{n-1} (R_{a^c} + T_{a^c})R_a = (R_{a^c} + T_{a^c})^{n-1} T_{a^c}R_a.$$
Since $T_{a^c}R_a$ 
is compact on $\L^2(\pi)$ and the set of compact operators on $\L^2(\pi)$    
is an ideal,  $(R_{a^c} + T_{a^c})^n R_a$ is compact on $\L^2(\pi)$. 
\end{proof}
\begin{lem} \label{lem-dec-op}
For every $n\geq 1$, there exists a compact operator $K_n$ on $\L^2(\pi)$ such that 
$$P^n = K_n + R_a^n + (R_{a^c} + T_{a^c})^n.$$
\end{lem}
\begin{proof}{}
For $n=1$ we have $P = K_1 + R_a + (R_{a^c} + T_{a^c})$ with $K_1:= T_a$ compact by Lemma~\ref{Ta-comp}. Now assume that the conclusion of Lemma~\ref{lem-dec-op} holds for some $n\geq 1$. Since 
the set of compact operators on $\L^2(\pi)$ forms a two-sided operator ideal, we obtain the following equalities for some compact operator $K'_{n+1}$ on $\L^2(\pi)$: 
\begin{eqnarray*}
P^{n+1} = P^nP &=& \big(K_n + R_a^n + (R_{a^c} + T_{a^c})^n\big) \big(K_1 + R_a + (R_{a^c} + T_{a^c})\big) \\
&=& K'_{n+1} + R_a^{n+1} + R_a^nR_{a^c} + R_a^nT_{a^c} + (R_{a^c} + T_{a^c})^nR_a + (R_{a^c} + T_{a^c})^{n+1} \\
&=& \big[K'_{n+1} + (R_{a^c} + T_{a^c})^nR_a\big] + R_a^{n+1} + (R_{a^c} + T_{a^c})^{n+1}.
\end{eqnarray*}
 Then the expected conclusion holds true for $P^{n+1}$ since $K_{n+1} := K'_{n+1} + (R_{a^c} + T_{a^c})^nR_a$ is compact on $\L^2(\pi)$ from Lemma~\ref{Ta-comp}.
\end{proof}

Theorem~\ref{theo-r-ess} is deduced from the next proposition which states that $\|T_{a^c}\|_2 \leq \beta_a$. Indeed, observe that $\|R_a\|_2\leq r_a$ and $\|R_{a^c}\|_2 \leq r_a'$. Set $\alpha_a:=\max\big(r_a\, , \, r_a'+\beta_a\big)$. Then Lemma~\ref{lem-dec-op} and $\|T_{a^c}\|_2 \leq \beta_a$ give 
$$\|P^n - K_n\|_2 \leq \|R_a\|_2^n + \big(\|R_{a^c}\|_2 + \|T_{a^c}\|_2\big)^n \leq 2\, \alpha_a^n.$$
The expected inequality $r_{ess}(P) \leq \alpha_a$ in Theorem~\ref{theo-r-ess} then follows from Formula~(\ref{def-r-ess}). 
\begin{pro} \label{lem-norm-alpha}
For any $a >0$, we have $\|T_{a^c}\|_2 \leq \beta_a$. 
\end{pro}
\begin{proof}{}
Lemma~\ref{lem-ineg} below shows that, for any bounded and continuous function $f : \R\r\C$, we have $\|T_{a^c}f\|_{\L^2(\pi)} \leq \beta_a\|f\|_{\L^2(\pi)}$. Then Inequality $\|T_{a^c}\|_2 \leq \beta_a$ of Proposition~\ref{lem-norm-alpha} follows from a standard density argument using that $\|T\|_2 \leq \|P\|_2 =1$ and that the space of bounded and continuous functions from $\R$ to $\C$ is dense in $\L^2(\pi)$.  
\end{proof}
\begin{lem} \label{lem-ineg}
For any bounded and continuous function $f : \R\r\C$, we have 
\begin{gather} 
\|T_{a^c}f\|_{\L^2(\pi)} \leq \int_{-s}^s\bigg[\int_{\{|x|>a\}}  |f(x+u)|^2\, t(x,x+u)^2\, \pi(x)\, dx\bigg]^{\frac{1}{2}}\, du\leq \beta_a\|f\|_{\L^2(\pi)}. \label{ineg-lem-ineg-1}
\end{gather}
\end{lem}
\begin{proof}{}
 Let $f : \R\r\C$ be a bounded and continuous function. Set $B:=\R\setminus[-a,a]$. Then it follows from (\ref{Ass-sup-fini}) that 
\begin{equation} \label{def-T-ac}
(T_{a^c}f)(x) = 1_{B}(x)\int_\R f(y)\, t(x,y)\, dy = 1_{B}(x)\int_{-s}^s f(x+u)\, t(x,x+u)\, du.
\end{equation}
For $n\geq 1$ and for $k=0,\ldots,n$, set $u_k:=-s + 2sk/n$ and define the following functions: $h_k(x) := 1_{B}(x)\, f(x+u_k)\, t(x,x+u_k)$. Then 
\begin{eqnarray} \label{lim-riem-1}
\bigg[\int_B\bigg|\frac{2s}{n} 
\sum_{k=1}^n h_k(x)\bigg|^2\pi(x)\, dx\bigg]^{\frac{1}{2}} & = & \bigg\|\frac{2s}{n} 
\sum_{k=1}^n h_k\bigg\|_{\L^2(\pi)} \leq \frac{2s}{n} \sum_{k=1}^n \|h_k\|_{\L^2(\pi)} \nonumber \\
& \leq &\frac{2s}{n} \sum_{k=1}^n \bigg[\int_B  |f(x+u_k)|^2\, t(x,x+u_k)^2\, \pi(x)\, dx\bigg]^{\frac{1}{2}}.
\end{eqnarray}
Below we prove that, when $n\r+\infty$, the left hand side of (\ref{lim-riem-1}) converges to $\|T_{a^c}f\|_{\L^2(\pi)}$ and that the right hand side of (\ref{lim-riem-1}) converges to the right hand side of the first inequality in (\ref{ineg-lem-ineg-1}). Define 
$$\forall x\in B,\quad \chi_n(x) := \frac{2s}{n} \sum_{k=1}^n h_k(x) = \frac{2s}{n} \sum_{k=1}^n f(x+u_k)\, t(x,x+u_k).$$ 
From Riemann's integral it follows that   
$$\forall x\in B,\quad \lim_{n\r+\infty}  \chi_n(x) = \int_{-s}^s f(x+u)\, t(x,x+u)\, du$$
since the function $u\mapsto f(x+u)\, t(x,x+u)$ is continuous on $[-s,s]$  from the assumptions of Theorem~\ref{theo-r-ess}. Note that $\sup_n\sup_{x\in B}|\chi_n(x)| < \infty$ since $f$ and $t$ are bounded functions. From Lebesgue's theorem and from (\ref{def-T-ac}), it follows that 
\begin{eqnarray} 
\lim_{n\r+\infty}\int_B\bigg|\frac{2s}{n} 
\sum_{k=1}^n h_k(x)\bigg|^2\pi(x)\, dx &=&
\int_B\bigg|\int_{-s}^s f(x+u)\, t(x,x+u)\, du\bigg|^2\pi(x)\, dx \nonumber \\
&=& \|T_{a^c}f\|_{\L^2(\pi)}^2. \label{lim-riem-2}
\end{eqnarray}
Next, observe that 
$$\frac{2s}{n} \sum_{k=1}^n \bigg[\int_B  |f(x+u_k)|^2\, t(x,x+u_k)^2\, \pi(x)\, dx\bigg]^{\frac{1}{2}} = \frac{2s}{n} \sum_{k=1}^n \psi(u_k)$$
with $\psi$ defined by  
$$\psi(u) := \bigg[\int_B  |f(x+u)|^2\, t(x,x+u)^2\, \pi(x)\, dx\bigg]^{\frac{1}{2}}. $$
Using the assumptions Theorem~\ref{theo-r-ess}, it follows from Lebesgue's theorem that $\psi$ is continuous. Consequently Riemann integral gives  
\begin{equation} \label{lim-riem-3}
\lim_{n\r+\infty} \frac{2s}{n} \sum_{k=1}^n \bigg[\int_B  |f(x+u_k)|^2\, t(x,x+u_k)^2\, \pi(x)\, dx\bigg]^{\frac{1}{2}} 
= \int_{-s}^s \psi(u)\, du.
\end{equation}
The first inequality in (\ref{ineg-lem-ineg-1}) follows from (\ref{lim-riem-1}) by using (\ref{lim-riem-2}) and (\ref{lim-riem-3}). 

Let us prove the second inequality in (\ref{ineg-lem-ineg-1}). The detailed balance equation (\ref{bal-eq}) gives
\begin{eqnarray*}
\lefteqn{\int_{-s}^s\bigg[\int_{\{|x|>a\}} |f(x+u)|^2\, t(x,x+u)^2\, \pi(x)\, dx\bigg]^{\frac{1}{2}}\, du } \\
& =  &  \int_{-s}^s\bigg[\int_{\{|x|>a\}} |f(x+u)|^2 \, t(x,x+u) \, t(x,x+u)\, \pi(x)\, dx\bigg]^{\frac{1}{2}}\, du  \\
&= &  \int_{-s}^s\bigg[\int_{\{|x|>a\}}  \, t(x,x+u) \, t(x+u,x)\, |f(x+u)|^2\ \pi(x+u)\, dx\bigg]^{\frac{1}{2}}\, du \\
&\leq & \quad \|f\|_{\L^2(\pi)}\int_{-s}^s\sup_{|x| > a}\sqrt{t(x,x+u)\, t(x+u,x)}\, du =  \|f\|_{\L^2(\pi)} \beta_a.
\end{eqnarray*}
\end{proof}
\section{Conclusion}
The study of the iterates of a Metropolis-Hasting kernel $P$ is of great interest to estimate the numbers of iterations required to achieve the convergence in the Metropolis-Hasting algorithm. In conclusion we discuss this issue by comparing the expected results  depending on whether $P$ acts on $\cB_V$ or on $\L^2(\pi)$. Recall that the $V$-geometrical ergodicity for $P$ (see Remark~\ref{rem-V-geo}) writes as: there exist  $\rho\in(0,1)$ and $C_\rho>0$ such that 
	\begin{equation} \label{ineg-gap-V} 
	\forall n\geq1,\ \forall f\in\cB_V,\quad  \|P^nf-\pi(f)\|_V \leq  C_\rho\, \rho^n\, \|f\|_V. \tag{SG$_V$}
\end{equation}
Let $\varrho_V(P)$ be the infinum bound of the real numbers $\rho$ such that (\ref{ineg-gap-V}) holds true. 
\begin{enumerate}
	\item In most of cases, the number $\varrho_V(P)$ is not known for Metropolis-Hasting kernels. The upper bounds of $\varrho_V(P)$ derived from drift and minorization inequalities seem to be poor and difficult to improve, excepted in stochastically monotone case (e.g.~see \cite[Sec.~6]{MenTwe96} and \cite{Bax05}). Consequently the inequality $\varrho_2(P) \leq \varrho_V(P)$ (see \cite[Th.~6.1]{Bax05}) is not relevant here. Observe that applying the quasi-compactness approach on $\cB_V$ would allow us to estimate the value of $\varrho_V(P)$, but in practice this method cannot be efficient since no accurate bound of the essential spectral radius of $P$ on $\cB_V$ is known.  
		\item The present paper shows that considering the action on $\L^2(\pi)$ rather than on $\cB_V$ of a Metropolis-Hasting kernel $P$ enables us to benefit from the richness of Hilbert spaces. The notion of Hilbert-Schmidt operators plays an important role for obtaining our bound (\ref{70}). The reversibility of $P$, that is $P$ is self-adjoint on $\L^2(\pi)$, implies that any upper bound $\rho$ of $\varrho_2(P)$ gives the inequality $\|P^nf - \Pi f\|_2 \leq \rho^{\, n}\, \|f\|_2$ for every $n\geq1$ and every $f\in\L^2(\pi)$. Consequently any such $\rho$ provides  an efficient information to estimate the numbers of iterations required to achieve the convergence in the Metropolis-Hasting algorithm. 
\end{enumerate}
From our bound (\ref{bound-r-ess-L2}) it can be expected that the quasi-compactness method (cf.~Introduction) will give a numerical procedure for estimating $\varrho_2(P)$ in the continuous state space case.

\bibliographystyle{alpha}

\end{document}